\documentclass[11pt]{article}
\usepackage{amsmath, amsthm, hyperref, xcolor, amsfonts, amssymb}

\usepackage{enumitem, csquotes}
\usepackage{parskip} 

\theoremstyle{plain}
\newtheorem{theorem}{Theorem}

\newtheorem{prop}[theorem]{Proposition}
\newtheorem{cor}{Corollary}[theorem]
 
\theoremstyle{definition}
\newtheorem*{definition}{Definition}
\newtheorem*{remark}{Remark}
\newtheorem{eg}{Example}
\hypersetup{                   
    colorlinks=true,                
    breaklinks=true,                
    urlcolor= black,                
    linkcolor= blue,                
    bookmarksopen=false,
    filecolor=black,
    citecolor=blue,
    linkbordercolor=blue
}
\DeclareMathOperator{\cp}{cp}

\newcommand{\nm}{\trianglelefteq}

\newcommand{\Z}{\mathbb{Z}}
\newcommand{\N}{\mathbb{N}}
\newcommand{\bP}{\mathbb{P}}
\newcommand{\comment}[1]{}
\newcommand{\commentvspace}[1]{}
\newcommand{\cC}{\mathcal{C}}
\hypersetup{                   
    colorlinks=true,                
    breaklinks=true,                
    urlcolor= black,                
    linkcolor= blue,                
    bookmarksopen=false,
    filecolor=black,
    citecolor=blue,
    linkbordercolor=blue
}

\comment{\setlist[itemize]{itemsep=0pt, parsep=0pt}
\setlist[enumerate]{itemsep=0pt, parsep=0pt}}

\begin{document}
\title{Commuting Probability of Finite Groups (Extended)}
\author{Snehinh Sen\\Indian Statistical Institute, Bangalore Centre\\snehinh1math@gmail.com}

\maketitle

\begin{abstract}
The target of this article is to discuss the concept of \textit{commuting probability} of finite groups which, in short, is a probabilistic measure of how abelian our group is. We shall compute the value of commuting probability for many special classes of non-abelian groups and also establish some local and global bounds. We will conclude with a few topics for further reading. 
\end{abstract}

\comment{\textbf{MSC:} 11A07, 12E20, 13M05}

\textit{ORCiD:} 0000-0002-7423-0090

\section*{Introduction}

\textit{Commuting probability} is a way of stating ``how abelian'' a group is. It is a natural numerical measure used to answer the question ``when do two elements of a group commute?'' As abelian groups are easier to study, one might then try to use probabilistic methods to prove or disprove some facts about pretty complicated groups by the use of commuting probability. The basic notion was introduced and studied in \cite{erd68}, \cite{gust73} and \cite{jos77}.

The target of this article is to summarise certain known results with proofs accessible to undergraduate students familiar with basic group theory. We will also try to improve certain known results or give an alternate approach on some instances. To keep up with this spirit, the proofs of most of the results are included. Yet certain results, which are undoubtedly worth mentioning, have rather advanced or long proofs. We omit proofs in such cases and provide appropriate references for the interested readers. 

It is very important to note that \textit{commuting probability} is not the only measure of how close to being abelian our group is. Few other measures, after normalisation are - the size of the center (a global measure using subgroups), the sizes of centralisers (a local measure using subgroups), size of the abelianization (a measure using quotients) and the class equation (a measure using conjugacy classes). As we go along, we will try to see how these different measures correlate to commuting probability.

Here is how the rest of the article is organised. The section \textit{Primary Considerations} \ref{sec1} introduces definitions, few basic results and examples. The next section  \textit{The Dihedral, The Symmetric and The Alternating} will focus on explicit computations for these special classes of groups. In the following section \textit{Bounding the Commuting Probability}, we shall establish some global and local bounds, including the very famous \textit{Erdős 5-8 Theorem}. Finally, \textit{Further Adventures} is a selected catalogue of topics for further study. Except the last section, all groups are assumed to be finite unless mentioned otherwise.

To the best of the author's knowledge, certain results presented here have not appeared in the given form elsewhere. These results are Propositions \ref{2}, \ref{20}, \ref{21}, Theorems \ref{19}, \ref{22}  and Corollaries \ref{18.1}, \ref{19.1}. Some of the proofs also differ from the sources. Another key aspect of this article is an intuitive reinterpretation of commuting probability as an \textit{antitone} (order reversing) \textit{information number} of a group, that is, in a very vague and intuitive sense, we argue that, in a fixed set-up, groups with larger commuting probability contain ``lesser (commuting) information''.  

\section{Primary Considerations}\label{sec1}

Let us start with a few definitions. As our group is finite, the most natural probability measure should be the one where elements are chosen uniformly at random. So suppose $G\times G$ is assigned with the discrete uniform distribution. Let $L(G)$ be the event $L(G)=\{(x,y)\in G\times G: xy=yx\}$. So $L(G)$ is the set of all pairs of commuting elements. Commuting probability should thus be the probability of this event $L(G)$ occurring in $G$. 

\begin{definition}
Let $G$ be a finite group. We define the \emph{commuting probability} of $G$ as 
\[\cp(G):= \bP (xy=yx : x, y \in G) = \frac{|L(G)|}{|G\times G|}=\frac{|L(G)|}{|G|^2}\]
where $L(G)=\{(x,y)\in G\times G: xy=yx\}$ is the event that an arbitrary pair $(x,y)\in G\times G$ commutes. 
\end{definition} 

\begin{remark}
Let $g,h\in G$. The \textit{commutator} $[g,h]$ is defined as $g^{-1}h^{-1}gh \in G$. It is called so because $[g,h]=1$ if and only if $g,h$ commute. So $L(G) = \{(x,y)\in G\times G : [x,y]=1\}$ is an alternate definition.

Given a group $G$, the \textit{commutator subgroup} or \textit{derived subgroup} $G'=[G,G]$ is defined as the subgroup generated by all the commutators $[g,h]$. 

It turns out to be normal and $Ab(G)=G/G'$ is the \textit{largest} abelian quotient of $G$ and is thus called its \textit{abelianization}.
\end{remark}

To lay the foundations, let us see how \textit{abelian-ness} translates for the aforementioned measures and try to give an elementary bound for commuting probability.
\begin{prop}
\label{1} 
The following are equivalent for a group $G$.
\begin{enumerate}[label=(\roman*)]
    \item $G$ is abelian.
    \item $Z(G)=G$, where $Z(G)$ is the center of $G$.
    \item $\cp(G)=1$ or equivalently $L(G) = G\times G$.
    \item $Z_G(a)=G$ for each $a\in G$, where $Z_G(a)$ is the centraliser of $a$, which is, by definition, $\{y\in G : ay=ya\}$.
    \item The class equation of $G$ is $1+1+\ldots +1$. That is, $\cC_a=\{a\}$ for each $a\in A$, where $\cC_a$ is the conjugacy class of $a$.
    \item $G'=1$ or equivalently $Ab(G)=G$.
\end{enumerate}
\end{prop}

\begin{prop}\label{2}
For any non-trivial group $G$, that is $|G|\ge 2$, we have $\cp(G)\ge \frac{3|G|-2}{|G|^2}$. Moreover, if $|G|\geq 3$, $\cp(G)\geq \frac{3}{|G|}$.
\end{prop}
 \commentvspace{-10pt}
\begin{proof}
Observe that for each $g\in G$, $(1,g), (g,1), (g,g)\in L(G)$, so $|L(G)|\ge 2|G|-1+|G|-1=3|G|-2$. Hence, $\cp(G)\ge \frac{3|G|-2}{|G|^2}$. Now if $|G|\geq 3$ and $G$ is non-abelian, then there must exist an element $a$ of order at least $3$. Then $(a, a^2)$ and $(a^2, a) \in G$. Hence, $|L(G)|\geq 3|G|$, giving us the desired result. 
\end{proof}

It can be noted that this is surely not the best of bounds. We will come to this later. For now, let us see some examples.

\begin{eg} Look at $S_3=<x,y|x^3=y^2=1=(xy)^2=1>$, the smallest non-abelian group. Very explicitly calculating:  
    \begin{align*}
    L(G) &= (\{1\}\times G) \cup (G\times \{1\})\cup \{(y,y),(xy,xy),(x^2y,x^2y)\} \\
        &\cup \{(x^k, x^l) : 1\leq k, l\leq 2\}.
    \end{align*}
    So $|L(G)|=18$. Thus, $\cp(S_3)=\frac{1}{2}$. \end{eg}
     \commentvspace{-5pt}

\begin{eg} Let $Q_8$ be the group of quaternions. Again, one may explicitly calculate that $|L(G)|=40$. So here we have $\cp({Q_8})=\frac{5}{8}=0.625>0.5$. 
\end{eg}

So, for the sake of it, one might say that \textit{`` even though $Q_8$ and $S_3$ are both non-abelian, $Q_8$ is 'more abelian' than $S_3$''}.

We now relate centralisers, hence conjugacy classes, to commuting probability. The following two results will be key to our analysis. We follow \cite{gust73}.

\begin{prop}\label{3}
For any group $G$, $|L(G)|=\sum_{x\in G} |Z_G(x)|$ (where $Z_G$ is the centraliser of $x$ in $G$). 
\end{prop}
 \commentvspace{-10pt}
\begin{proof}
Note that $L(G)=\{(x,y)\in G^2: xy=yx\}=\coprod_{x\in G}\{x\}\times Z_G(x)$. So $|L(G)|=\sum_{x\in G} |Z_G(x)|$.
\end{proof}

\commentvspace{2mm}

\begin{theorem}[{\cite[Theorem~IV]{erd68}} or  {\cite[p.~1031]{gust73}}]
\label{4}
For any group $G$, if $K$ is the number of conjugacy classes (that is, the \textit{class number} of $G$), then $\cp(G)=\frac{K}{|G|}.$
\end{theorem}
 \commentvspace{-10pt}
\begin{proof}
From group theory, we have that $|Z_G(a)||\cC_a|=|G|$ for every $a\in G$. Also, conjugacy classes partition $G$. Hence, we have 
\[ K = \sum_{g\in G} \frac{1}{|\cC_g|} = \sum_{g\in G} \frac{|Z_G(g)|}{|G|}\]
whence, by Proposition \ref{3}, our claim follows. 
\end{proof}
 \commentvspace{-10pt}
Let us apply Theorem \ref{4} to some small non-abelian groups.
\begin{itemize}
    \item In $S_3$, there are 3 conjugacy classes, so $\cp(S_3)$ is $1/2$.
    \item In $Q_8$, there are 5 conjugacy classes, so $\cp(Q_8)$ is $5/8$.
    \item In $A_5$, there are 5 conjugacy classes and 60 elements, so $\cp(A_5)$ is $1/12$ (without doing 3600 multiplications).
\end{itemize}

\section{The Dihedral, The Symmetric and The Alternating}
\label{sec2}
In this section, we will try to compute the commuting probability of some standard classes of groups. Before proceeding further, here are two standard results (see, for eg. \cite[p.~120,126]{dum03}) from group theory which we will need in this section. 

\begin{prop}\label{5}
For $S_n$, two elements are conjugates if and only if they have the same cycle type. Further, if $G\leq S_n$ then two conjugates have the same cycle type when considered in $S_n$.
\end{prop} 
\begin{prop}[Cayley's Theorem] \label{6}
Every finite group $G$ is contained in the symmetric group $S_{|G|}$. Specifically, $D_{2n}, A_n\leq S_n$.
\end{prop}

We start by considering \textit{Dihedral Groups} of $2n$ elements, $D_{2n}$. The author was introduced to these results by Professor B. Sury.

\begin{prop}\label{7}
Let $G=D_{2n}$ where $n\geq 3$. Then $\cp(G)$ is $\frac{n+6}{4n}$ if $n$ is even and $\frac{n+3}{4n}$ if $n$ is odd.
\end{prop}

 \commentvspace{-10pt}
\begin{proof}
We just prove for the case when $n$ is odd. The even case is similar. Let $D_{2n}=<x, y | x^n=1, y^2=1, (xy)^2=1>$. Then observe that for each $1\leq p \leq n-1$, $Z_G(1)=G$, $Z_G(x^p)=<x>$, $Z_G(y)=\{1,y\}$ and $Z_G(x^py)=\{1,x^py\}$. Hence, by Proposition \ref{3}, we have 
\[\cp(G)=\frac{\sum_{x\in G}|Z_G(x)|}{|G|^2}=\frac{1\cdot2n+(n-1)\cdot n + n\cdot 2}{4n^2}=\frac{n+3}{4n}.\]\end{proof}

It is clear that the sequence of probabilities $\cp(D_{2n})$, $n\geq 3$, has alternate crests and troughs and converges to $0.25$. Furthermore, $\cp(D_{2n}) \leq 0.5$ for each $n\neq 4$.

\begin{remark}\label{r2} 
Beyond $n=3$, we get that $D_{2n}$ is non-abelian. A small checking would show that $\cp(D_{2n}) \leq \frac{5}{8}$ for each $n\geq 3$ with equality only for $n=4$. In fact, for any non-abelian group $G$ discussed so far, we had $\cp(G)\le \frac{5}{8}$. We shall soon see that this is indeed a global upper bound.
\end{remark}

We shift our focus to $S_n$. A \textit{partition} of a natural number $n$ is an unordered collection of natural numbers $a_1,\ldots,a_l$ which add up to $n$. $p(n)$ would denote the number of partitions. For example, as $4=4=1+3=2+2=1+1+2=1+1+1+1$ we would have $p(4)=5$. Observe that number of conjugacy classes of $S_n$ = the number of cycle types in $S_n$ = the number of \textit{partitions} of $n$. So we have -
\commentvspace{-6pt}

\begin{prop} \label{8} 
The number of conjugacy classes of $S_n$ is equal to $p(n)$, the number of partitions of $n$. Therefore, $\cp(S_n) = \frac{p(n)}{n!}$.
\end{prop}
\commentvspace{-8pt}

\begin{remark}\label{r3} 
The first few values of commuting probability of $S_n$ - for $n=3,4,5,6$ they are respectively $0.5, \frac{5}{24}, \frac{7}{120}, \frac{11}{720}$. As can be seen, this decreases rapidly.
\end{remark}

Even though there are many known approximations and neat series which asymptotically converge to $p(n)$ (check for example \cite{ram18} for more details), there is no known closed formula for this function. Here is a well-known simple upper bound. 

\begin{prop}
\label{9}
For any natural number $n$, $p(n)\le 2^{n-1}$.
\end{prop}
\commentvspace{-10pt}
\begin{proof}
 Any partition is a solution to the equation $x_1+....+x_k=n$ with each $x_i\geq 1$ for some $k=1,...,n$. Total number of solutions of such equations is $2^{n-1}$. Hence, $p(n)\le 2^{n-1}$. 
\end{proof}
\commentvspace{-6pt}

As $n$ becomes larger, $\cp(S_n)$ goes to zero. Hence, commuting probability has no non-trivial global lower bound.

We conclude this section by analysing the alternate group $A_n$. We give a formula using different types of partitions.

\begin{definition}
Let $n\in \N$. An \textit{odd distinct partition} (ODP) of $n$ is a partition of $n$ consisting of odd and distinct parts. The corresponding cycle type is called an \textit{odd distinct cycle type} (ODC).
\end{definition}

It is helpful to recall that the sign of a permutation  is only dependent on its cycle type. Here is a well-known result characterizing the conjugacy classes of $A_n$. For example, one might refer to \cite{planmathconj}. 

\begin{theorem}
\label{10}
A conjugacy class $\cC$ of $S_n$ with cycle type $t$ of even permutations remains unchanged in $A_n$ if and only if there is an odd permutation $p$ and a permutation $x\in \cC$ such that $xp=px$. Moreover, this happens if and only if $t$ is not ODC. If $t$ is ODC, then $\cC$ splits into two identical parts.
\end{theorem}
\commentvspace{-5pt}
\begin{proof}

Suppose $x\in A_n$. Let $\cC_x$ and $\cC_x$ denote its conjugacy class in $S_n$. By Proposition \ref{5}, we have $\cC'_x \subseteq \cC_x$. So $[A_n : Z_{A_n}(x)] \leq [S_n : Z_{S_n}(x)]$. Moreover, $Z_{A_n}(x) \leq Z_{S_n}(x)$. Thus, $[Z_{S_n}(x) : Z_{A_n(x)}] \leq 2$ with equality if and only if $\cC'_x = \cC_x$. So $\cC'_x=\cC_x$ if and only if $Z_{A_n}(x) < Z_{S_n}$ which is true if and only if there is an odd permutation $p$ commuting with $x$. Otherwise, it will split into exactly two equally sized classes in $A_n$. 

We now wish to see how this relates to ODC. Suppose $x\in \cC$ is not ODC. Then either $x$ has an even cycle or two identical odd cycles. In the first case, this even cycle, call it $p$, isan odd permutation in the centraliser. In the other case, if the two cycles of the same size are $(a_1,\ldots, a_k)$ and $(b_1,\ldots, b_k)$, take $p=(a_1,b_1)\ldots (a_k,b_k)$. Clearly, $p$ is odd and in $Z_{S_n}(x)$.

Conversely, if $x$ is ODC, then let us denote its cycle decomposition (including singleton, if any) by $C_1C_2\ldots C_k$ where $n_i=|C_i|$ and $n_1<n_2<\ldots n_k$. Then clearly $|\cC_x|=\frac{n!}{n_1n_2\ldots n_k}$. Thus $|Z_{S_n}(x)|=n_1 n_2 \ldots n_k$. Now consider the subgroup $H=<C_1,\ldots, C_k>$. Then $|H|=n_1n_2\ldots n_l$ and $H\le Z_{S_n}(x)$. So $H=Z_{S_n}(x)$. But $H\le A_n$. Thus, $Z_{A_n}(x)=Z_{S_n}(x)$ and $\cC_x$ splits in $A_n$.
\end{proof}

Let $q(n)$ denote the number of ODPs of $n$. Let $r(n)$ and $s(n)$ respectively denote the number of partitions of $n$ with even many even parts and odd many even parts. Then using formal power series manipulations, one can directly show that $r(n)-s(n)=q(n)$ and $r(n)+s(n)=p(n)$ for each $n\geq 1$. So, to summarise, we have the following result.

\begin{cor}
\label{10.1}
Let $G=A_n$, then $\cp(A_n)=\frac{2(r(n)+q(n))}{n!}=\frac{p(n)+3q(n)}{n!}$.
\end{cor}

\begin{eg}\label{ealternating}
For $n=4$, $r(4)=3$ and $q(4)=1$. So commuting probability is $\frac{1}{3}$. Likewise, for $n=5,6$ we get the probabilities are $\frac{1}{12}$ and $\frac{7}{360}$.
\end{eg}

\section{Bounding the Commuting Probability}
\label{sec3}
In this section, we shall be computing some global and local bounds on commuting probability. We start by recalling a few basic results [See, for eg. \cite[p.~84-89]{dum03}] from group theory. 

\begin{prop}\label{11}
Let $G, H$ be groups and $Z(G)$ be the center of $G$.
\begin{enumerate}
    \item $G/Z(G)$ is cyclic if and only if $G=Z(G)$.
    \item For any $(a,b)\in G\times H$, $Z_{G\times H}(a,b) = Z_G(a)\times Z_H(b)$.
\end{enumerate} 
\end{prop} 
\commentvspace{-6pt}
An immediate consequence of the above is the following. One may use it and the groups $G_n = S_3\times S_3 \ldots S_3$ ($n$ times) and give an alternate proof of the fact that commuting probability has no lower bound. 

\begin{prop}\label{12} 
If $G$ and $H$ are two finite non-abelian groups. Then $\cp({G\times H}) = \cp(G) \times \cp(H)$. 
\end{prop}

Earlier on, we observed that for small non-abelian groups $\cp(G)\leq \frac{5}{8}$. We are derive the famous Erd\"os 5-8 Theorem, which confirms our observations, and give a group theoretic corollary.

\begin{theorem}[Erdős 5-8 Theorem, see for eg. {\cite[p.~1032]{gust73}}] 
\label{13}
Let $G$ be a finite non-abelian group. Then $\cp(G)\leq \frac{5}{8}$. Moreover, equality holds for infinitely many groups.
\end{theorem} 

\commentvspace{-10pt}
\begin{proof}
Let $G$ be a non-abelian group. Then by Proposition \ref{11}, $[G:Z(G)]\geq 4$. Moreover, if $a\notin Z(G)$, then $[G:Z_G(a)]\geq 2$. So by Proposition \ref{3}, we get -
\begin{align*}
\cp(G) = \sum_{g\in G} \frac{|Z_G(g)|}{|G|^2} 
	  &= \sum_{g\in Z(G)} \frac{|Z_G(g)|}{|G|^2} + \sum_{g\in G\setminus Z(G)} \frac{|Z_G(g)|}{|G|^2} \\
	  &\leq \sum_{g\in Z(G)} \frac{|G|}{|G|^2} + \sum_{g\in G\setminus Z(G)} \frac{1}{2|G|} \\
	  &= \frac{|Z(G)|}{|G|} + \frac{|G|-|Z(G)|}{2|G|} \\
	  &= \frac{1}{2} + \frac{|Z(G)|}{2|G|} \leq \frac{5}{8} 
\end{align*}
Finally, observe that for any abelian group $H$, $\cp({H\times Q_8})$ is indeed $5/8$. This concludes the proof.\end{proof} 
\commentvspace{-10pt}

There are several interesting applications of the 5-8 theorem, for example, one can bound the number of order $2$ elements in a non-abelian group $G$. A proof would require some character theory. Interested reader are referred to Corollary $3.1$ and Lemma $2$ of \cite{man94}.

\begin{cor} 
\label{13.1}
Any non abelian finite group $G$ has at most $\left\lfloor\frac{5|G|}{8}\right\rfloor$ conjugacy classes, where $\lfloor . \rfloor$ is the floor function. 
\end{cor}

A natural attempt would be to categorise all groups for which equality holds in Theorem \ref{13}. Such groups are called \textit{5-8 groups}. Note that equality holds if and only if 
\begin{enumerate}[label=(\roman*)]
\item $[G:Z(G)]=4$, and
\item $[G:Z_G(y)] = 2$ for each $y\in G\setminus Z(G)$, that is every non-trivial conjugacy class has $2$ elements. 
\end{enumerate}

However, observe that $(1)$ implies $(2)$ as well as the fact $G/Z(G) \cong V_4$, the Klein $4-$group. Thus $G$ is a 5-8 group if and only if $G/Z(G)$ is isomorphic to $V_4$ which is if and only if $[G:Z(G)]=4$. A better characterization is hinted in Section $3$ of \cite{gust73}. 

Note that the bound can be slightly improved in the case when the smallest prime dividing $|G|$ is $p>2$. This is implicit in the above proof. Further local improvement is also possible. 

\begin{theorem}\label{14}
Let $G$ be a finite non-abelian group with $p$ being the smallest prime dividing $|G|$. Then 
\[\cp(G)\leq \frac{1}{p} + \frac{(p-1)}{p[G:Z(G)]} \leq \frac{p^2+p-1}{p^3}\]
All three are equal if and only if $G/Z(G) \cong \Z/p\Z \times \Z/p\Z$.
\end{theorem}
\commentvspace{-10pt} 
\begin{proof}
The proof follows from the proof of Theorem \ref{13} realising that if $a\notin Z(G)$, then $[G:Z_G(a)]=|G|/|Z_G(a)| \geq p$. For the second inequality, note that $[G:Z(G)]\geq p^2$ by Proposition \ref{11}. Equality case is similar to above discussions.
\end{proof}
\commentvspace{-10pt}

\begin{remark}\label{r4} 
A very large set of groups for which equality holds is $G\cong P\times H$ where $H$ is abelian and $P$ is a $p-$group for which $[G:Z(G)]=p^2$. In fact, any $5-8$ group is of this form. Let $a\in H$ be an element such that $(o(a), p) = 1$, then $a\in Z(G)$. Let $H:=\{a\in G : (o(a), p) = 1\}$. Then $H\leq Z(G)$ and if $S$ is a Sylow $p-$group of $G$, we get $HS = G$. Together, this implies that $G=H\times S$ and $\cp(S)=\cp(G) = \frac{5}{8}$. One can hence show that $S$ is precisely a non-abelian central extension of $\Z/p\Z \times \Z/p\Z$.
\end{remark}

In \cite[p.~202]{mac74}, it is mentioned that if $G$ is finite and $\cp(G)>\frac{1}{2}$, then $\cp(G)$ is of the form $\frac{1}{2} + \frac{1}{2^{2s+1}}$ where $s\geq 0$. A proof can be found in \cite{paul95}. This can be used to refine our bound as follows.

\begin{prop}\label{15} 
Let $G$ be a non-abelian group of order $n$. If $8$ does not divide $n$, then we have $\cp(G) < \frac{5}{8}$. More precisely, $\cp(G) \leq \frac{1}{2}$. Equality holds for infinitely many groups.
\end{prop}
\commentvspace{-10pt}
\begin{proof}
Observe that, by the statement preceding the proposition, if $\cp(G) > \frac{1}{2}$, then $\cp(G) = \frac{m}{2^{2k+1}}$, where $m$ is odd and $k\geq 1$. Let $K$ be the class number of $G$. Then $\frac{K}{n} = \cp(G) = \frac{m}{2^{2k+1}}$. Thus, $mn = 2^{2k+1}K$.  As $m$ is odd and $k\geq 1$, we have $8$ divides $n$. Considering the contrapositive, if $n$ is not divisible by $8$, we thus get $\cp(G) \leq \frac{1}{2}$, as desired. Equality will hold whenever $G\cong S_3 \times H$, where $H$ is an odd abelian group. 
\end{proof}
\commentvspace{-10pt}

Perhaps some effort can be made to classify all equality cases. For example, see \cite{paul95}. As we stated earlier, greater commuting probability indicates ''lesser commuting information'', because, in a sense, abelian groups of a given order contain the least amount of information due to commutativity. So it is natural to guess that subgroups and quotients contain lesser information (due to their derived nature) than the ambient group. Indeed this is true!

In fact, a much stronger result holds (see Theorem \ref{18}). An analysis of these results can be found in \cite{gal70}. We give a proof for the weaker cases, namely quotients and subgroups.

\begin{theorem}\label{16} 
Suppose $G$ is a finite group and $H\nm G$.  Then $\cp(G)\leq \cp(G/H)$. Equality holds if and only if $[x,y]\in H$ implies $[x,y] =1$. Thus, if equality holds, then $H\leq Z(G)$. 
\end{theorem}
\commentvspace{-10pt}
\begin{proof}
Let $\bar{G}=G/H$. Note that 
\commentvspace{-8pt}
\begin{align*}
L({\bar{G}}) &= \{(xH,yH) \in \bar{G}\times \bar{G} : [xH, yH] = 1H\}\\ 
			&= \{(xH,yH) \in \bar{G}\times \bar{G} : [x, y] \in H\}.
\end{align*}
Thus, $|H|^2|L({\bar{G}})| = |\{(x,y)\in G\times G : [x,y]\in H\}|\geq |L(G)|$, from where our result follows. 

Equality holds if and only if $\{(x,y) \in G\times G : [x, y] \in H\} = L(G)$, that is to say, $[x,y]\in H \implies [x,y]=1$. Moreover, $x\in H$ and $H\nm G$ implies $[x,y]\in H$ for each $y\in G$. So $[x,y]=1$, implying that $H\leq Z(G)$. This completes our proof.
\end{proof}
\commentvspace{-10pt}

Note that $H\leq Z(G)$ is not sufficient for equality - take $G$ to $Q_8$ and $H=Z(G)$. Then $H\leq Z(G)$ and $\cp(G/H) =1$ but $\cp(G)<1$.

\begin{theorem}\label{17}
Suppose $G$ is a finite group and $H\leq G$. Then $\cp(H) \geq \cp(G)$.
\end{theorem}
\commentvspace{-10pt}
\begin{proof}
Observe that for each $h\in H$, $Z_H(h)=Z_G(h)\cap H$. In general, for a $g\in G$, let $Z_H(g) = Z_G(h)\cap H$. By Proposition 13 of \cite[p.~93]{dum03}. We get
\[ |Z_G(g)\cap H| = \frac{|Z_G(g)||H|}{|Z_G(g)H|} \geq \frac{|Z_G(g)||H|}{|G|}.\]
Set $m=[G:H]$. Thus, we have $m|Z_H(g)|\geq |Z_G(g)|$. Also, by double counting, we get
\[\sum_{g\in G} |Z_H(g)| = |\{(g,h) : g\in G, h\in H, gh=hg\}|=\sum_{h\in H} |Z_G(h)|.\]
Therefore {\small
\[ |L(G)| = \sum_{g\in G} |Z_G(g)|\leq \sum_{g\in G} m |Z_H(g)| =\sum_{h\in H} m |Z_G(h)| =\sum_{h\in H} m^2 |Z_H(h)| \]}
which would directly imply our result.
\end{proof}

We record a stronger result without any proof (see, for eg., \cite{gal70}) and give some corollaries following \cite{gur06}. 

\begin{theorem}\label{18} 
Suppose $G$ is a finite group. Let $H\nm G$. Then 
\commentvspace{-10pt} \[ \cp(G) \leq \cp(H) \cp(G/H). \]
\end{theorem}
\commentvspace{-10pt}
\begin{cor}\label{18.1}
Let $G, H$ be two groups and suppose $A=G\ltimes H$ is a semi-direct product of these groups (for example, see \cite[p.~175]{dum03}. Then $\cp(A) \leq \cp(G)\cp(H)$.
\end{cor}
\commentvspace{-10pt}
\begin{proof}
This is true as $H\nm A$ and $A/H \cong G$.  
\end{proof}
\commentvspace{-10pt}

Intuitively, semi-direct product has more commutative information (the joining map) compared to the direct product with the same underlying sets. For example, treating $H, K \leq H\times K$ via natural inclusions, $hk=kh$ for each $h\in H$ and $k\in K$. This is not the case in semi-direct products. So this corollary should follow from our intuition and Proposition \ref{12}.

\begin{cor}\label{18.2}
Let $1=G_0\nm G_1 \ldots \nm G_k = G$ be a composition series of a group. Let $H_i = G_i/G_{i-1}$ be the $i^{th}$ composition factor. Then $\cp(G)\leq \prod_{i=1}^{k} \cp(H_i)$.
\end{cor}

\begin{proof}
Follows from induction on $k$ and Theorem \ref{18}.
\end{proof}

As we had noted earlier, it is not possible to find a global non-trivial lower bound. However, just like Proposition \ref{2}, we can define some lower bounds depending on the properties of $G$. We improve Theorem 2.1 of \cite{bas17} in the setting of groups. 

\begin{theorem}[Group version of {\cite[Theorem~2.1]{bas17}}]\label{19}
 Suppose $G$ is a finite group. Let $p$ be the smallest prime dividing $|G|$. Let $m=[G:Z(G)]$. Then we have
\[\cp(G)\geq \frac{(p+1)m-p}{m^2}.\]
{\small Equality holds if and only if $[Z_G(a):Z(G)]=p$ for each $a\notin Z(G)$.}
\end{theorem}

\commentvspace{-10pt}  

\begin{proof}
Observe that for each $a\in G\setminus Z(G)$, $p|Z(G)|\leq |Z_G(a)|$. Using this, we have
\commentvspace{-4pt}
\begin{align*}
|L(G)| = \sum_{g\in G}|Z_G(g)| &=|G||Z(G)|+\sum_{g\in G\setminus Z(G)}|Z_G(g)| \\
&\geq|G||Z(G)|+ \sum_{g\in G\setminus Z(G)} p|Z(G)|\\
&=|G||Z(G)|+p(|G|-|Z(G)|)|Z(G)|
\end{align*}
\commentvspace{-4pt}

whence the given inequality follows. Equality holds if and only if $p|Z(G)|= |Z_G(a)|$ for each $a\notin Z(G)$.
\end{proof}

\begin{remark}\label{r5} 
Once again, it can be shown that any group for which equality holds is of the form $P\times H$, where $H$ is an abelian group and $P$ is a p-group with the aforementioned equality. For if $a\in G$ has $(o(a),p)=1$, we must have $a\in Z(G)$. Otherwise, as $[Z_G(a):Z(G)] = p$ and $a\in Z_G(a)$, by taking quotients we would get $p|o(a)$. Rest of the proof is similar to Remark \ref{r5}. Once again, one can try to characterize all such $p-$groups. 
\end{remark}

Let $G$ be a non-abelian group. Observe that $m\geq p^2>p$. Using this, we get $\cp(G) > p/m$. Now $\cp(G)=K/|G|$, where $K$ is the class number of $G$. So, we have $Km>p|G|$, that is $K>p|Z(G)|$. Using this, we have this pretty fascinating group theoretic result.

\begin{cor}\label{19.1}
Let $G$ be a finite non-abelian group of order $n$ and let $p$ be the smallest prime dividing $n$, then the number of conjugacy classes of $G$ is at least $p|Z(G)|+1$. Hence, there are at least $(p-1)|Z(G)|+1\geq p$ many non-trivial conjugacy classes.
\end{cor}

Till now, we tried to look at the size and prime factorization of the size of the groups to bound commuting probability. One could also study specific classes of group. We now try to formulate some results specifically about simple groups via elementary methods. A proper study would once again require advanced tools like representation theory, which we do not intend to use. Nevertheless, we give ample references for the interested readers. 

The first simple non-abelian groups has order $60$. Beyond that, all simple non-abelian groups have order $|G|\geq 168$ and, according to \cite{conrconj}, at least $6$ conjugacy classes. We shall make use of this fact. Here is a group theoretic result which we will need.

\begin{prop}\label{20} 
Let $G$ be a non-abelian simple group of order $n \geq k!$. Then $n$ has no subgroup $H$ of index $[G:H]\leq k$. Thus, every conjugacy class has size at least $k+1$. 
\end{prop} 
\commentvspace{-15pt}
\begin{proof}
Note that it is enough to show that there is no subgroup of index $k$. Suppose, on the contrary, there is a subgroup $H$ with index $k$. Let $L$ be the set of left cosets of $H$. Then $G$ acts on $L$ via left multiplication. Using this, we get a homomorphism $\phi : G\to S_k$. As $G$ is simple, we must have $Ker(\phi)=1$. But then $|G|\geq k!$. Thus we must have $\phi$ is an isomorphism, which would contradict that $G$ is simple. \end{proof}
\commentvspace{-10pt}
\begin{prop}\label{21}
For any simple non-abelian group $G$, $\cp(G) < \frac{1}{5}$.
\end{prop}
\commentvspace{-12pt}
\begin{proof}
As $G$ is simple, $G$ has a non-trivial center. Suppose $K$ is the number of conjugacy classes of $G$ and $k$ is the size of the smallest non-trivial conjugacy class of $G$. Then by considering the average size of the non-trivial conjugacy classes, we get $k(K-1)\le (|G|-1)$. Now $K(|G|+5)> 6|G|$ which would say $k<\frac{6|G|}{5K}$. But by Proposition \ref{20}, we must have $k\geq 6$. Thus $\cp(G) =\frac{K}{|G|} < \frac{1}{5}$ for each simple $G$ with order at least $168$. But the only simple group of order $60$ is $A_5$, which has commuting probability $\frac{1}{12}<\frac{1}{5}$. Therefore, our bound holds for every simple group.
\end{proof}
\commentvspace{-15pt}
One can make this bound considerably better with some representation theory. Here is the strongest possible bound.

\begin{theorem}[Dixon]\label{22}
Let $G$ be a simple non-abelian group. Then $P_G\leq \frac{1}{12}$. Equality holds only for $A_5$. 
\end{theorem}

A proof can be found on \cite[p.~302]{can73} (as a problem due to J.Dixon) and uses facts from representation theory and matrix groups. We state a fascinating corollary. The proof follows from Theorem \ref{22}, Corollary \ref{18.2} and the Jordan-H\"older Theorem for groups.

\begin{cor}\label{22.1}
Every finite group $G$ with $\cp(G)>\frac{1}{12}$ is solvable.
\end{cor}

For an alternate proof and a description of equality, readers are referred to Theorem 11 of \cite{gur06}. To conclude our discussion on simple groups, we look at the following remarkable result. It can be found in \cite{harden} in the comments by I.Agol and D.L.Harden.\comment{It is well known that number of finite groups with a given number of conjugacy classes is finite. This result is a slight variation of the same in the setting of commuting probability and simple groups. }

\begin{theorem} \label{23}
Let $\epsilon>0$. Then the number of simple finite non-abelian groups with $\cp(G)\geq \epsilon$ is finite.
\end{theorem}
\commentvspace{-10pt}
\begin{proof}
We follow the notation used in the proof of Proposition \ref{21}. Let $|G|\geq 168$. Now if $\cp(G)\geq \epsilon$, then $k<\frac{6}{5\epsilon}$. By Proposition \ref{20}, $k! > |G|$. So if $m=\left\lfloor\frac{6}{5\epsilon}\right\rfloor$, we should have $|G|<m!$ and clearly there are finitely many such groups.
\end{proof}
\commentvspace{-10pt}
We conclude this section by trying to relate commuting probability to the derived subgroup $G'$. Recall that the larger $G'$ is, the farther away $G$ is from being abelian. In fact, with some elementary representation theory, one can relate these two quantities pretty easily. For example, see the appendix of \cite{bro19}. We record the result without a proof along with an obvious improvement. This would be followed by a bound on the other side.

\begin{theorem}\label{24} 
Let $G$ be a finite group, then 
\[\cp(G) \leq \frac{1}{4} \left( 1+ \frac{3}{|G'|}\right).\]
In fact, if $p$ is the smallest prime dividing the order of $G$, then the above bound can be slightly improved to 
\[ \cp(G) \leq \frac{1}{p^2} \left( 1+ \frac{p^2-1}{|G'|}\right).\]
\end{theorem} 

\begin{remark}\label{r6} 
If $G$ is a non-abelian finite group, then $|G'|\geq p$, which would recover the $5-8$ bound and Theorem \ref{14} using Theorem \ref{24}.
\end{remark}

\begin{prop}[Group version of {\cite[Theorem~2.5]{bas17}}]\label{25} 
Let $G$ be a finite group. Then 
\[\cp(G)\geq \frac{[G:Z(G)]+|G'|-1}{|G'|[G:Z(G)]}\] with equality if and only if $|G'|=|\cC_g|$ for each $g\notin Z(G)$.
\end{prop}
\commentvspace{-10pt}
\begin{proof}
Observe that the function $\phi : \cC_g \to G'$ given by $h\mapsto g^{-1}h$ is an injection, hence $|G'|\geq |\cC_g|=[G:Z_G(g)]$. Rest follows from the proof of Theorem \ref{19}.
\end{proof}
\commentvspace{-10pt}

\section{Further Adventures}\label{sec4}
Having developed quite some background about commuting probability, one might ask - what else? This was just the tip of the iceberg. Here are a few paths in which one could proceed.

\textbf{Topological Properties} Define $\mathcal{P}:=\{\cp(G) : |G|<\infty\}$. It is clear that $\mathcal{P}\subseteq (0,0.625]\cup \{1\}$. According to \cite{mac74}, there are quite a few ``gaps'' in this set. It is fairly obvious that the derived set $\mathcal{P}'$ contains $0$ and is a subset of $[0,0.625]$. What else is there in $\mathcal{P}'$? Such questions were first asked by Keith Joseph in \cite{jos77}, who proposed that $\cp$ is a naturally well ordered set (using $>$) and $\bar{\mathcal{P}} = \{0\} \cup \mathcal{P}$. This is a pretty amazing claim! However, as one might expect, the proof is pretty complicated and, to the knowledge of the author, certain parts are yet to be proven. A suitable reference is \cite{sean15}. 

\textbf{Infinite Groups} Suppose $G$ is a (locally) compact topological group. Over here, the \textit{Haar measure} can be used to define a commuting probability. As indicated by \cite{gust73}, the $5-8$ Theorem, with a small modification, is still valid in such a set-up. One may then proceed to ask other questions about commuting probability. A few of these results have been considered in \cite{toit20}. 

\textbf{Other Algebraic Structure} Instead of looking at groups, one could venture into the realm of finite rings, algebras, non-associative rings and so on. For example, one may start with \cite{bas17}. 

\textbf{Isoclinism} Isoclinism is a phenomena introduced by Philip Hall to classify $p-$groups. It is a generalisation of isomorphism of groups. Recall that we have a \textit{commutator map}  
$\phi : G/Z(G) \times G/Z(G) \to G'$ given by $(aZ(G), bZ(G)) \mapsto  [a,b]$. We say two groups $G_1$ and $G_2$ are \textit{isoclinic} if there commutator maps are, effectively, the same. That is, we have (a) $G_1/Z(G_1) \cong G_2/Z(G_2)$ via some $\psi$ (b) $G_1' \cong G_2'$ via some $\theta$ (c) If $\phi_i$ is the commutator map of $G_i$, then $\theta \circ \phi_1 = \phi_2  \circ \psi \times \psi$ as maps. So the isomorphisms commute with commutator maps. Remarkably, if two groups are isoclinic, they have the same commuting probability. A sample reference for such considerations is \cite{paul95}. 

\textbf{Other Probabilities} There are many more interesting probabilities on a group $G$ of which we list a few. Let $n\geq 2$.
\begin{enumerate}
\item Probability that an arbitrary $n-$tuple in $G$ commutes, that is \[\bP \left((g_1,\ldots, g_n) \in G^n : \prod_{i=1}^n g_i = \prod_{i=1}^n g_{\sigma(i)}, \forall \sigma \in S_n\right).\] 
\item Probability that $n$ randomly chosen elements generates $G$.
\item Commuting probability of a subgroup with respect to a group, that is, $\bP ((g, h) \in G\times H : gh=hg)$.
\item Probability that two arbitrarily selected elements are conjugates, or in general, satisfy some group theoretic property. 
\end{enumerate}

\section*{Acknowledgement}
The author would like to thank the professors and students of ISI Bangalore for inspiring many aspects of this article. A special thanks to Prof. Yogeshwaran D., Prof. Parthanil Roy and Prof. B. Sury for their suggestions to improve the write-up of the article.


\begin{thebibliography}{1}

\bibitem{bro19} Browning, T. \textit{Commuting Probability} (note) \url{https://bit.ly/2Zccmx9}.

\bibitem{conrconj} Conrad, K. \textit{Conjugation in a Group} (note) \url{https://bit.ly/3vSBd4G}.

\bibitem{can73} Dixon, J, Problem 176, page 302. \textit{Canad. Math. Bull.}, 1973, 16(2). 

\bibitem{dum03}  Dummit, D.S. and Foote, R.M., (2003) \textit{Abstract Algebra}, Wiley. \comment{ISBN
9780471433347}

\bibitem{bas17} Dutta, J. and Basnet, D.K. (2017), Some bounds for commuting probability of finite rings. \textit{Proc. Indian Acad. Sci. Math. Sci.}

\bibitem{sean15} Eberhard, S. (2015), Commuting probabilities of finite groups. \textit{Bull. Lond. Math. Soc.} 47(5).

\bibitem{erd68} Erd\"os  P.  and Turan P. (1968), On some problems of a statistical group-theory, IV. \textit{Acta Math. Hungar.} 19  413--435.

\bibitem{gal70} Gallagher, P. X. (1970), The Number of Conjugacy Classes in a Finite Group. \textit{Math. Z.} 118 : 175-179.

\bibitem{gur06} Guralnick, R. M., Robinson,G.R. (2006),
On the commuting probability in finite groups. \textit{J. Algebra}, 300(2)

\bibitem{gust73} Gustafson, W. H. (1973), What is the Probability that Two Group Elements Commute? \textit{Amer. Math. Monthly } 80(9).

\bibitem{ram18} Hardy, G. H. and Ramanujan, S. (1918), Asymptotic Formulae in Combinatory Analysis. \textit{Proc. Lond. Math. Soc.} 17, 75-115.

\bibitem{jos77} Joseph, K. (1977), Several conjectures on commutativity in algebraic structures. \textit{Amer. Math. Monthly }, 84:550–551.

\bibitem{paul95} Lescot, P.(1995) Isoclinism Classes and Commutativity Degrees of Finite Groups. \textit{J. Algebra}, 177, 847-869.

\bibitem{mac74} MacHale, D.(1974), How Commutative Can a Non-Commutative Group Be? \textit{Math. Gaz.}, 58(405).

\bibitem{man94} Mann, A.(1994), Finite Groups Containing Many Involutions. \textit{Proc. Amer. Math. Soc.} 122(2) pp. 383-385.

\bibitem{toit20} Tointon, M.C.H (2020), Commuting Probability of Infinite Groups. \textit{J. Lond. Math. Soc.} 101(3).

\bibitem{planmathconj} PlanetMath, Conjugacy in $A_n$, \url{https://planetmath.org/conjugacyinan}

\bibitem{harden}Discussion between D.L. Harden and I.Agol on MathOverflow 
\url{https://bit.ly/3D5N9my}


\end{thebibliography}
\end{document}